\documentclass[a4paper,12pt,reqno]{amsart}

\usepackage{amsmath}
\usepackage{amscd}
\usepackage{amssymb}
\usepackage{mathrsfs}
\usepackage{stmaryrd}
\usepackage{pst-node}
\usepackage{tikz-cd}
\usepackage[left=2.5cm,right=2.5cm,bottom=3cm,top=3cm]{geometry}

\newtheorem{theorem}{Theorem}[section]

\newtheorem{lemma}{Lemma}[section]
\newtheorem{corollary}{Corollary}[section]

\newtheorem{remark}{Remark}[section]

\newcommand{\del}{\nabla}
\newcommand{\laplace}{\triangle}

\newcommand{\ulp} {\underline{\partial}}

\newcommand{\pa} {{\partial_0 }}
\newcommand{\bfe} {{\mathbf e}}

\newcommand{\NSc} {\text{NSc}}

\newcommand{\ulx} {{\underline x}}

\def\be{\begin{equation}}
	\def\ee{\end{equation}}

\def\endproof{{\hfill $\square$}\medskip}

\AtEndDocument{\bigskip{\footnotesize
		\textsc{Macao Center for Mathematical Sciences, Macau University of Science and Technology, Macao, China}\par
		\textit{E-mail address: }\texttt{wxmai@must.edu.mo}\par
		\medskip
		\textsc{School of Mathematical Sciences, Xiamen University, Xiamen, 361005, China}\par
		\textit{E-mail address: }\texttt{oujianyu@xmu.edu.cn}
		
}}

\begin{document}
	\title{Three Balls Theorem for Eigenfunctions of Dirac Operator in Clifford Analysis}
	
	\author{Weixiong Mai \& Jianyu Ou$^*$}\thanks{$^*$ Corresponding author}
	\begin{abstract}
	In this paper we establish the three balls theorem for functions $u$ satisfying $Du=\lambda u$ in Clifford analysis, where $D$ is the Dirac operator. As an application, we generalize Hadamard's three circles theorem to monogenic function in $\mathbb R^{n+1}.$\\
	
	
	\end{abstract}

	\maketitle

	\section{Introduction}	
	The famous Hadamard three circles theorem for holomrophic functions is stated as follows:
	\begin{theorem}[Hadamard,1896]
		Let $f$ be a holomorphic function on the annulus $\{z\in \mathbb C: r_1\leq|z|\leq r_3\}$ with  $0<r_1<r_2<r_3<\infty.$ Denote by $M(r)$ the maximum of $|f(z)|$ on the circle $|z|=r.$ Then
		\begin{align*}
			\left\{M(r_2)\right\}^{\log{\frac{r_3}{r_1}}}\leq \left\{M(r_1)\right\}^{\log{\frac{r_3}{r_2}}}\left\{M(r_3)\right\}^{\log{\frac{r_2}{r_1}}}.
		\end{align*}
	\end{theorem}

	\noindent The theorem was originally announced by J. Hadamard in \cite{Hadamard}. Its standard proof could be found in e.g. \cite{Duren,Protter-Weinberger}.  It could be shown that the three circles theorem also holds for harmonic functions, as well as subharmonic functions, in $n$-dimensional Euclidean spaces (see e.g. \cite{Protter-Weinberger}). The majority of proofs of Hadamard's three circles theorem is usually based on the commutativity in the algebra of holomorphic functions.
	
	In recent years, Hadamard's three circles theorem has been generalized to solutions of various partial differential equations by the frequency function method, which was first introduced by F. Almgren in \cite{Almgren}. The frequency function method leads to an $L^2$-version of Hadamard's three circles (or balls) theorem. It has been shown that the frequency function method and Hadamard's three circles theorem have powerful applications in the study of unique continuation for elliptic equations (see e.g. \cite{Garofalo-Lin}). Hence, Hadamard's three circles theorem has many interesting applications in partial differential equations and differential geometry (see e.g. \cite{Garofalo-Lin,Garofalo-Lin1,Zhu,Alessandrini-Rondi-Rosset-Vessella,Xu,Liu,Colding-Minicozzi, Ou} and the references therein).
	
	In this paper we consider a generalization of three circles theorem about Dirac operator in Clifford analysis. Let $\{\bfe_1,...,\bfe_n\}$ be an orthonormal basis of $\mathbb R^n$ associated with the rule
	\begin{align*}
		\bfe_i\bfe_j+\bfe_j\bfe_i=-2\delta_{ij},\quad 1\leq i,j\leq n,
	\end{align*}
	where $\delta_{ij}$ is the Kronecker symbol. Denote by $ \mathbb R^{(n)}$ the Clifford algebra generated by $\{\bfe_1,...,\bfe_n\}$ over $\mathbb R,$ whose elements are of the form $x= \sum_{A} x_A\bfe_{A} $, where $A=\{1\leq j_1<j_2<\cdots<j_l\leq n\}$ runs over all ordered subsets of $\{1,...,n\},x_A\in\mathbb R$ and $x_\emptyset=x_0$ with the identity element $\bfe_\emptyset=\bfe_0=1$ (see \S 2 for a brief introduction to Clifford algebra). The Dirac operator is defined as 
	\begin{align*}
		D = \partial_0 +\sum_{j=1}^n \partial_j\bfe_j = \frac{\partial}{\partial x_0} + \sum_{A}\frac{\partial}{\partial x_j}\bfe_j.
	\end{align*}
	Suppose that $u(x)=\sum_{A}u_A(x)\bfe_A,$ and $\lambda\in\mathbb R.$ Then, $u$ is said to be an eigenfunction of $D$ if 
	\begin{align}\label{Du}
		D u = \lambda u,
	\end{align}
	when $\lambda \neq 0.$ When $\lambda =0$, (\ref{Du}) means that $u$ is left-monogenic, which generalizes the concept of holomorphic function to $\mathbb R^{n+1}$. Therefore, it is natural and significant to study whether there holds a three balls theorem for monogenic functions taking values in $\mathbb R^{(n)}$. Moreover, we will establish the three balls theorem for (\ref{Du}) with general $\lambda$ in this paper. Due to the noncommutativity of Clifford algebra, the standard argument of proving Hadamard's three circles theorem is invalid in our case (even for $\lambda=0$). In \cite{Abul-Constales-Morais-Zayed} Abul-Ez, Constales, Morais and Zayed proved a three balls theorem for the so-called special monogenic functions, where the used technique only works for their special case. The three lines theorem in Clifford analysis is given by Peetre and Sj\"olin in \cite{Peetre-Sjolin}. We also note that in the complex case the three circles theorem can be deduced by the three lines theorem, where the property of complex exponential function plays a role (see \cite[pages 386-387]{Ullrich}). However, this argument is also invalid in the Clifford algebra setting. Therefore, in this paper we will adapt the frequency function method to obtain the $L^2$-version of three balls theorem for $u$ satisfying (\ref{Du}). Consequently, we can deduce the $L^\infty$-version of three balls theorem by the Moser iteration method, which is a powerful tool in the theory of elliptic partial differential equations (see e.g. \cite{Han-Lin}). In particular, the used frequency function was proposed by Zhu in \cite{Zhu} (see also {\cite{Ou}}), which is different from the one given in \cite{Garofalo-Lin}.
	
	Our main results are stated as follows.
	\begin{theorem}\label{three-balls-L2}
		Suppose $u=\sum_{A}u_A\bfe_A$ satisfies $Du=\lambda u$, and $\alpha\geq2$. Then, there exist $0<r_1<r_2<2r_2<r_3<\infty$ such that\\
		\noindent when $\lambda\neq 0$
		\begin{align}\label{3balls-1}
			\int_{B_{r_2}} |u(x)|^2 dx \leq C_3 \left(\int_{B_{r_1}} |u(x)|^2dx\right)^{\frac{C_1}{C_1+C_2}} \left(\int_{B_{r_3}} |u(x)|^2dx\right)^{\frac{C_2}{C_1+C_2}},
		\end{align}
		where $C_1=\frac{1}{\log\frac{2r_2}{r_1}}, C_2=\frac{1}{\log \frac{r_3}{2r_2}}$ and \\$C_3=\frac{r_1^{2\alpha\frac{C_1}{C_1+C_2}}r_3^{2\alpha\frac{C_2}{C_1+C_2}}}{3^{\alpha}r_2^{2\alpha}}e^{\frac{\left(\frac{a}{2}(r_3^2-(2r_2)^2)+b(r_3-2r_2)\right)}{(\alpha+1)C_2(C_1+C_2)}-\frac{\left(\frac{a}{2}((2r_2)^2-r_1^2)+b(2r_2-r_1)\right)}{(\alpha+1)C_1(C_1+C_2)}}$ with $a=\frac{2|\lambda|^2+|\lambda|}{3},b=\frac{5(\alpha+1)}{3}|\lambda|-\frac{2|\lambda|+1}{9}$ and $c=\frac{2(\alpha+1)(\alpha+n)}{3}-\frac{5(\alpha+1)}{18}+\frac{2|\lambda|+1}{54|\lambda|};$\\
		\noindent when $\lambda =0$
		\begin{align}\label{3balls-2}
			\int_{B_{r_2}} |u(x)|^2 dx \leq C_4\left(\int_{B_{r_1}}|u(x)|^2dx\right)^\frac{C_1}{C_1+C_2}\left(\int_{B_{r_3}}|u(x)|^2dx\right)^\frac{C_2}{C_1+C_2},
		\end{align}
		where $C_1,C_2$ are given as above, and $C_4=\frac{r_1^{2\alpha\frac{C_1}{C_1+C_2}}r_3^{2\alpha\frac{C_2}{C_1+C_2}}}{3^{\alpha}r_2^{2\alpha}}.$
	\end{theorem}
	
	For the case $\lambda=0$, (\ref{3balls-2}) is the $L^2$-version of the three balls inequality for monogenic functions. Then, by the subharmonic inequality we can deduce the following result.
	\begin{corollary}\label{three-spheres-infty}
		Suppose $u=\sum_{A}u_A\bfe_A$ satisfies $Du=0$, and $\alpha\geq 2$. Then, there exist $0<r_1<r_2<2r_2<r_3<\infty$ such that
		\begin{align}\label{3balls-infty}
			||u||_{L^\infty(B_{r_2})} 
			\leq  3^\frac{n}{2}C_4^\prime (r_3-2r_2)^{-\frac{n}{2}}r^\frac{n}{2}_3||u||_{L^\infty({B_{r_1}})}^{\frac{C_1^\prime}{C_1^\prime+C_2^\prime}} ||u||_{L^\infty(B_{r_3})}^{\frac{C_2^\prime}{C_1^\prime+C_2^\prime}},
		\end{align}
		where $C_1^\prime=\frac{1}{\log \frac{2(r_3+r_2)}{3r_1}},$ $C_2^\prime= \frac{1}{\log \frac{3r_3}{2(r_3+r_2)}}$ and $C_4^\prime=\frac{3^\alpha r_1^{2\alpha\frac{C_1}{C_1+C_2}}r_3^{2\alpha\frac{C_2}{C_1+C_2}}}{4^\alpha((r_3+r_2))^{2\alpha}}.$
	\end{corollary}
	
 \begin{remark}
 Corollary \ref{three-spheres-infty} is a generalization of Hadamard's three circles theorem in the setting of Clifford algebra. On one hand, Corollary \ref{three-spheres-infty} holds for general monogenic functions, which is more general than that proposed in \cite{Abul-Constales-Morais-Zayed}. On the other hand, since the constant on the right-hand side of (\ref{3balls-infty}) is bigger than one, in this sense our result is weaker than the classical Hadamard's three circles theorem. It would be interesting whether there holds a complete generalization of the classical Hadamard's three circles theorem for general monogenic functions. 
\end{remark}

	It seems that the approach in this paper could not lead to this complete generalization. Nevertheless, the frequency function method can be used to obtain a $L^\infty$-version of the three balls theorem for $\lambda\neq 0,$ which is stated as follows.
	
	\begin{corollary}\label{three-spheres-infty1}
		Suppose $u=\sum_{A}u_A\bfe_A$ satisfies $Du=\lambda u$ with $\lambda\neq 0$, and $\alpha\geq 2$. Then, there exist $0<r_1<r_2<2r_2<r_3<1$ such that
		\begin{align*}
			||u||_{L^\infty(B_{r_2})} \leq  M^\prime C_3^\prime (r_3-2r_2)^{-\frac{n}{2}}r^\frac{n}{2}_3||u||_{L^\infty({B_{r_1}})}^{\frac{C_1^\prime}{C_1^\prime+C_2^\prime}} ||u||_{L^\infty(B_{r_3})}^{\frac{C_2^\prime}{C_1^\prime+C_2^\prime}},
		\end{align*}
		where $M^\prime$ is a positive constant, $C_1^\prime=\frac{1}{\log \frac{2(r_3+r_2)}{3r_1}},$ $C_2^\prime= \frac{1}{\log \frac{3r_3}{2(r_3+r_2)}}$\\
		\noindent and $C_3^\prime=\frac{3^\alpha r_1^{2\alpha\frac{C_1}{C_1+C_2}}r_3^{2\alpha\frac{C_2}{C_1+C_2}}}{4^\alpha((r_3+r_2))^{2\alpha}}e^{\frac{\left(\frac{a}{2}(r_3^2-(\frac{2(r_3+r_2)}{3})^2)+b(r_3-\frac{2(r_3+r_2)}{3})\right)}{(\alpha+1)C_2(C_1+C_2)}-\frac{\left(\frac{a}{2}((\frac{2(r_3+r_2)}{3})^2-r_1^2)+b(\frac{2(r_3+r_2)}{3}-r_1)\right)}{(\alpha+1)C_1(C_1+C_2)}}.$
	\end{corollary}

	The paper is organized as follows. In \S 2 some basic notations and properties of Clifford algebra, and the frequency function are introduced. In \S 3 the monotonicity of frequency function $N(r)$ is proved. In \S 4 the main results are proved.

{\bf Acknowledgement} The first author was supported in part by NSFC Grant No. 11901594, and Guangdong Basic and Applied Basic Research Foundation Grant No. 2021A1515010351, and FRG Program of the Macau University of Science and Technology, No. FRG-22-076-MCMS. The second author was supported in part by the Natural Science Foundation of Fujian Province  (Grant No. 2022J05007) and the Fundamental Research Funds for the Central Universities (No. 20720220042).

	\section{Preliminaries}
	First we introduce some basic notations and properties of Clifford algebras.
	Let $\bfe_1,...,\bfe_n$ be basic elements satisfying
	$$
	\bfe_j \bfe_k+ \bfe_k \bfe_j =-2\delta_{jk}, \quad j,k=1,...,n,
	$$
	where $\delta_{jk}$ is the Kronecker delta function.
	Let $\mathbb R^n=\{\ulx =x_1\bfe_1+\cdots +x_n\bfe_n;x_j\in \mathbb R, 1\leq j\leq n\}$ be identical with the usual Euclidean space $\mathbb R^n,$ and $\mathbb R^{n+1}=\{x=x_0+\ulx:x_0\in\mathbb R,\ulx\in\mathbb R^{n}\}.$ 
	
	The real Clifford algebra $\mathbb R^{(n)}$ generated by $\{\bfe_1,...,\bfe_n\}$ is the associative algebra generated by $\{\bfe_1,...,\bfe_n\}$ over the real field $\mathbb R$.
	The elements of $\mathbb R^{(n)}$ are of the form $x=\sum_{A}x_A \bfe_A,$ where $A=\{1\leq j_1<j_2<\cdots<j_l\leq n\}$ runs over all ordered subsets of $\{1,...,n\}$, $x_A\in \mathbb R$ with $x_{\emptyset}=x_0,$ and $\bfe_A=\bfe_{j_1}\bfe_{j_2}\cdots \bfe_{j_l}$ with the identity element $\bfe_{\emptyset}=\bfe_0=1.$ Sc$(x):=x_0$ and NSc$(x):=x-\text{Sc}(x)$ are respectively called the scalar part and the non-scalar part of $x.$
	We denote the conjugate of $x\in\mathbb R^{(n)}$ by $\overline x=\sum_{A} x_A \overline \bfe_A$, where $\overline \bfe_A=\overline \bfe_{j_l}\cdots \overline \bfe_{j_2}\overline \bfe_{j_1}$ with $\overline \bfe_0=\bfe_0$ and $\overline \bfe_j=-\bfe_j$ for $j \neq 0.$ The norm of $x\in \mathbb R^{(n)}$ is defined as $|x|:=(\text{Sc } \overline x x)^\frac{1}{2}=(\sum_{A}|x_A|^2)^{\frac{1}{2}}.$  $x=x_0+\ulx\in \mathbb R^{n+1}$ is called a para-vector, and the conjugate of a para-vector $x$ is $\overline x=x_0-\ulx.$ If $x$ is a para-vector then $x^{-1}=\frac{\overline x}{|x|^2}.$ For more information about Clifford algebras, we refer to \cite{Brackx-Delanghe-Sommen}.
	
	The Dirac operator is defined as
	\begin{align*}
		D = \partial_0+\underline \partial=\partial_0+\sum_{j=1}^n\partial_j\bfe_j = \frac{\partial}{\partial x_0} +\sum_{j=1}^n \frac{\partial}{\partial x_j}\bfe_j.
	\end{align*}
	In Clifford analysis, the Dirac operator plays an important role since it gives rise to the development of monogenic function theory (see e.g. \cite{Brackx-Delanghe-Sommen}).
	
	In the following we introduce the frequency function used in this paper, which is proposed in \cite{Zhu}.
	Let $u=\sum_A u_A\bfe_A=u_0+\NSc (u)$ satisfy 
	\begin{align}\label{Du1}
		Du = \lambda u,
	\end{align}
	where $\lambda\in \mathbb R.$
	Note that
	\begin{align*}
		Du = & (\partial_0 +\ulp)(u_0+\NSc (u))\\
           = & \pa u_0 +\ulp u_0 +\pa \NSc(u) +\ulp \NSc(u)\\
           = & \lambda u\\
           = & \lambda u_0 +\lambda \NSc(u),
	\end{align*}
	which implies that $-\partial_0 u_0-\ulp\NSc(u) = \pa u_0+\pa\NSc(u) -\lambda u_0-\lambda \NSc(u).$
	Next we consider 
	\begin{align*}
		\laplace u &=\overline D Du\\ 
		&=\lambda (\pa -\ulp )(u_0+\NSc(u))\\
		&= \lambda ( \pa u_0 +\pa \NSc(u) -\ulp u_0-\ulp \NSc(u))\\
		&=\lambda (2\pa u_0 +2\pa \NSc(u) -\lambda u_0-\lambda \NSc(u)),
	\end{align*}
	which is equivalent to
	\begin{align}\label{lapU}
		\laplace u_A &= \lambda (2\pa u_A-\lambda u_A), \quad \text{for all possible index } A.
	\end{align}
	For $\alpha\geq 2,$ define 
	$$H(r)=\int_{B_r}|u|^2(r^2-|x|^2)^\alpha dx=\sum_A \int_{B_r}u_A^2(r^2-|x|^2)^\alpha dx=\sum_A\int u_A^2(r^2-|x|^2)^\alpha,$$
	where $B_r$ is the ball centered at the origin with radius $r.$
	For the sake of simplicity, we let $n_1=n+1$, and write $\int_{B_r}$ as $\int$ when there is no confusion.
	By taking the derivative for $H(r)$ with respect to $r$, we have
	\begin{align}\label{Hr}
		H^\prime(r)=\frac{2\alpha +n_1}{r}H(r) +\frac{1}{r(\alpha+1)} I(r),
	\end{align}
	where
	\begin{align}\label{Ir}
		I(r)=\int |\del u|^2(r^2-|x|^2)^{\alpha+1} dx+\sum_A \int u_A\laplace u_A (r^2-|x|^2)^{\alpha+1}dx.
	\end{align}
	The frequency function $N(r)$ is defined as 
	\begin{align*}
		N(r) = \frac{I(r)}{H(r)}.
	\end{align*}	
	\section{Monotonicity of frequency function }
	
	In the following lemma we prove a monotonic result for $N(r),$ which plays a crucial role in proving the main results.
	\begin{lemma}
		Suppose $u=\sum_{A}u_A\bfe_A$ satisfies $Du=\lambda u$. Then, for $\lambda \neq 0,$
		$e^{6|\lambda| r}(N(r)+p(r)) $ is nondescreasing, where
		$p(r)=ar^2+br+c$ with $a=\frac{2|\lambda|^2+|\lambda|}{3},b=\frac{5(\alpha+1)}{3}|\lambda|-\frac{2|\lambda|+1}{9}$ and $c=\frac{2(\alpha+1)(\alpha+n)}{3}-\frac{5(\alpha+1)}{18}+\frac{2|\lambda|+1}{54|\lambda|}.$ For $\lambda=0,$ we have that $N(r)$ is nondescreasing.
	\end{lemma}
	\begin{proof}

		Firstly, by taking the derivative for $I(r)$ in (\ref{Ir}) with respect to $r,$ we have
		\begin{align}\label{derI}
            \begin{split}
			I^\prime(r)  
            = & 2(\alpha+1)\int r|\del u|^2(r^2-|x|^2)^{\alpha}dx + 2(\alpha+1)\sum_A \int r u_A\laplace u_A (r^2-|x|^2)^\alpha dx\\
            = & \frac{2(\alpha+1)}{r}\int |\del u|^2(r^2-|x|^2)^{\alpha+1} - \frac{1}{r}\int |\del u|^2 \langle x,\del (r^2-|x|^2)^{\alpha+1}\rangle dx\\
            &+ 2(\alpha+1)\sum_A \int r u_A\laplace u_A (r^2-|x|^2)^\alpha dx.
            \end{split}
        \end{align}
        Using integration by parts, we have
        \begin{align*}
            I^\prime(r) = & \frac{2(\alpha+1)+n_1}{r}\int |\del u|^2(r^2-|x|^2)^{\alpha+1} + \frac{1}{r}\int \langle \del |\del u|^2,x\rangle  (r^2-|x|^2)^{\alpha+1} dx\\
            &+ 2(\alpha+1)\sum_A \int r u_A\laplace u_A (r^2-|x|^2)^\alpha dx.
        \end{align*}
        Note that
        \begin{align*}
           &\int \langle \del|\del u_A|^2,x\rangle (r^2-|x|^2)^{\alpha+1} dx \\
           = & 2\sum_{i,j}\int x_i\partial_j u_A \partial_i\partial_ju_A (r^2-|x|^2)^{\alpha+1} dx\\
           = & -2\int |\del u_A|^2(r^2-|x|^2)^{\alpha+1}dx -2\int \laplace u_A \langle \del u_A,x\rangle (r^2-|x|^2)^{\alpha+1}\\
           & + 4(\alpha+1)\int \langle \del u_A,x\rangle^2 (r^2-|x|^2)^\alpha dx. 
        \end{align*}
        Consequently, 
        \begin{align*}
            I^\prime(r) = & \frac{2\alpha+n_1}{r} \int |\del u|^2(r^2-|x|^2)^{\alpha+1} -\frac{2}{r}\sum_A\int \laplace u_A\langle \del u_A, x\rangle (r^2-|x|^2)^{\alpha+1}\\
            & +\frac{4(\alpha+1)}{r}\sum_A\int \langle \del u_A, x\rangle^2(r^2-|x|^2)^\alpha dx + 2(\alpha+1)\sum_A \int ru_A\laplace u_A(r^2-|x|^2)^\alpha dx. 
        \end{align*}
        
        Using (\ref{Ir}), we can rewrite $I^\prime(r)$ as
        \begin{align*}
			I^\prime(r)
			=&\frac{2\alpha+n_1}{r}I(r)+\frac{4(\alpha+1)}{r}\sum_A\int \langle \del u_A,x\rangle^2 (r^2-|x|^2)^\alpha dx\\
			&-\frac{2}{r}\sum_A\int \laplace u_A \langle \del u_A,x\rangle (r^2-|x|^2)^{\alpha+1}dx-\frac{n_1-2}{r}\sum_A\int u_A\laplace u_A (r^2-|x|^2)^{\alpha+1}dx\\
			& + \frac{2(\alpha+1)}{r}\sum_A \int |x|^2u_A \laplace u_A (r^2-|x|^2)^\alpha dx.
		\end{align*}
		We also note that
		\begin{align}\label{eq1}
			\begin{split}
				&-\frac{2}{r}\int \laplace u_A\langle \del u_A,x\rangle (r^2-|x|^2)^{\alpha+1}dx\\
				=&-\frac{2\lambda}{r}\int (2\partial_0 u_A - \lambda u_A)\langle \del u_A,x\rangle (r^2-|x|^2)^{\alpha+1}dx\\
				=&-\frac{4\lambda}{r}\int \partial_0u_A\langle \del u_A,x\rangle (r^2-|x|^2)^{\alpha+1}dx + \frac{2\lambda^2}{r}\int u_A\langle \del u_A,x\rangle (r^2-|x|^2)^{\alpha+1}dx\\
				\geq & -\frac{2|\lambda|}{r} \int |\partial_0u_A|^2|x|(r^2-|x|^2)^{\alpha+1}dx -\frac{2|\lambda|}{r}\int|\del u_A|^2|x|(r^2-|x|^2)^{\alpha+1}dx\\
				& - \frac{|\lambda|^2}{r}\int u_A^2 |x|(r^2-|x|^2)^{\alpha+1}dx -\frac{|\lambda|^2}{r}\int |\del u_A|^2 |x|(r^2-|x|^2)^{\alpha+1}dx\\
				\geq &-4|\lambda| \int |\del u_A|^2 (r^2-|x|^2)^{\alpha+1}dx -|\lambda|^2 \int u^2_A(r^2-|x|^2)^{\alpha+1}dx\\
                & -|\lambda|^2\int u_A^2(r^2-|x|^2)^{\alpha+1}dx,
			\end{split}
		\end{align}
		and
		\begin{align}\label{eq2}
				&-\frac{n_1-2}{r}\int u_A\laplace u_A(r^2-|x|^2)^{\alpha+1}dx\nonumber\\
				&=-\frac{(n_1-2)\lambda}{r}\int u_A(2\partial_0 u_A-\lambda u_A)(r^2-|x|^2)^{\alpha+1}dx\nonumber\\
				&=-\frac{2(n_1-2)\lambda}{r}\int u_A\partial_0u_A (r^2-|x|^2)^{\alpha+1}dx +\frac{(n_1-2)\lambda^2}{r}\int u_A^2 (r^2-|x|^2)^{\alpha+1}dx\nonumber\\
				&=-\frac{2(n_1-2)(\alpha+1)\lambda}{r}\int u_A^2 x_0 (r^2-|x|^2)^\alpha dx+\frac{(n_1-2)\lambda^2}{r}\int u_A^2(r^2-|x|^2)^{\alpha+1}dx\nonumber\\
				&\geq -2(n_1-2)(\alpha+1)|\lambda| \int u_A^2(r^2-|x|^2)^\alpha dx,
				\end{align}
		and
		\begin{align*}
				&\frac{2(\alpha+1)}{r}\int |x|^2u_A\Delta u_A (r^2-|x|^2)^\alpha dx\\
				= &\frac{2(\alpha+1)\lambda}{r}\int |x|^2u_A(2\partial_0u_A-\lambda u_A) (r^2-|x|^2)^\alpha dx\\
				= &\frac{2(\alpha+1)\lambda}{r} \int |x|^2 \partial_0u_A^2 (r^2-|x|^2)^\alpha dx -\frac{2(\alpha+1)\lambda^2}{r} \int |x|^2 u_A^2(r^2-|x|^2)^\alpha dx\\
				= &-\frac{4(\alpha +1)\lambda}{r}\int x_0 u_A^2(r^2-|x|^2)^\alpha dx+ \frac{4\alpha(\alpha+1)\lambda}{r}\int |x|^2x_0 u_A^2 (r^2-|x|^2)^{\alpha-1}dx\\
				& -\frac{2(\alpha+1)\lambda^2}{r} \int |x|^2 u_A^2(r^2-|x|^2)^\alpha dx.
			\end{align*}
			By simplifying the coefficients, we obtain
			\begin{align}\label{eq3}
			\begin{split}
			&\frac{2(\alpha+1)}{r}\int |x|^2u_A\Delta u_A (r^2-|x|^2)^\alpha dx\\
				\geq & -4(\alpha+1)|\lambda| \int u^2_A(r^2-|x|^2)^\alpha dx-\frac{2\alpha(\alpha+1)|\lambda|}{r}\int |x|^2|x_0|u_A^2(r^2-|x|^2)^{\alpha-1}dx\\
				&-\frac{2\alpha(\alpha+1)|\lambda|}{r}\int |x|^2|x_0| u_A^2(r^2-|x|^2)^{\alpha-1}dx-\frac{2(\alpha+1)|\lambda|^2}{r} \int |x|^2 u_A^2(r^2-|x|^2)^\alpha dx\\
				\geq & -4(\alpha+1)|\lambda| \int u^2_A(r^2-|x|^2)^\alpha dx- 4\alpha(\alpha+1)|\lambda| r^2\int u^2_A(r^2-|x|^2)^{\alpha-1}dx\\
				&-{2(\alpha+1)|\lambda|^2r} \int u_A^2(r^2-|x|^2)^\alpha dx.
			\end{split}
		\end{align}
		Combining (\ref{eq1}), (\ref{eq2}) and (\ref{eq3}), we have
		\begin{align*}
				&-\frac{2}{r}\sum_A\int \laplace u_A \langle \del u_A,x\rangle (r^2-|x|^2)^{\alpha+1}
				-\frac{n_1-2}{r}\sum_A\int u_A\laplace u_A (r^2-|x|^2)^{\alpha+1}dx \\
				&+ \frac{2(\alpha+1)}{r}\sum_A \int |x|^2u_A\laplace u_A (r^2-|x|^2)^\alpha dx  \\
				\geq & -4|\lambda| \sum_A\int |\del u_A|^2 (r^2-|x|^2)^{\alpha+1}dx -|\lambda|^2 \sum_A\int u^2_A(r^2-|x|^2)^{\alpha+1}dx\\
				&-|\lambda|^2\sum_A\int u_A^2(r^2-|x|^2)^{\alpha+1}dx-2(n_1-2)(\alpha+1)|\lambda| \sum_A\int u_A^2(r^2-|x|^2)^\alpha dx\\
				&-4(\alpha+1)|\lambda| \sum_A\int u^2_A(r^2-|x|^2)^\alpha dx- 4\alpha(\alpha+1)|\lambda| r^2\sum_A\int u^2_A(r^2-|x|^2)^{\alpha-1}dx\\
				 &-{2(\alpha+1)|\lambda|^2r} \int u_A^2(r^2-|x|^2)^\alpha dx\\
				 \end{align*}
\begin{align}\label{tem1}
\begin{split}
				= &-4|\lambda| \sum_A\int |\del u_A|^2 (r^2-|x|^2)^{\alpha+1}dx -4|\lambda|\sum_A\int u_A\laplace u_A (r^2-|x|^2)^{\alpha+1}dx\\
				&+4|\lambda| \sum_A \int u_A\laplace u_A (r^2-|x|^2)^{\alpha+1}dx-|\lambda|^2 \sum_A\int u^2_A(r^2-|x|^2)^{\alpha+1}dx\\
				&-|\lambda|^2\sum_A\int u_A^2(r^2-|x|^2)^{\alpha+1}dx-2(n_1-2)(\alpha+1)|\lambda| \sum_A\int u_A^2(r^2-|x|^2)^\alpha dx
				\\
				&-4(\alpha+1)|\lambda| \sum_A\int u^2_A(r^2-|x|^2)^\alpha dx- 4\alpha(\alpha+1)|\lambda| r^2\sum_A\int u^2_A(r^2-|x|^2)^{\alpha-1}dx\\
				&-{2(\alpha+1)|\lambda|^2r} \int u_A^2(r^2-|x|^2)^\alpha dx,
			\end{split}
		\end{align}
		which gives that
		\begin{align*}
			&(\ref{tem1})\\
            = &-4|\lambda| I(r)+8(\alpha+1)|\lambda|\lambda\sum_A\int u_A^2x_0(r^2-|x|^2)^\alpha dx -4|\lambda|^3\sum_A\int u_A^2(r^2-|x|^2)^{\alpha+1}\\
			&-|\lambda|^2 \sum_A\int u^2_A(r^2-|x|^2)^{\alpha+1}dx-|\lambda|^2\sum_A\int u_A^2(r^2-|x|^2)^{\alpha+1}dx\\
			&-2(n_1-2)(\alpha+1)|\lambda| \sum_A\int u_A^2(r^2-|x|^2)^\alpha dx
			-4(\alpha+1)|\lambda| \sum_A\int u^2_A(r^2-|x|^2)^\alpha dx\\
			&- 4\alpha(\alpha+1)|\lambda| r^2\sum_A\int u^2_A(r^2-|x|^2)^{\alpha-1}dx-{2(\alpha+1)|\lambda|^2r} \int u_A^2(r^2-|x|^2)^\alpha dx\\
			\geq & -4|\lambda| I(r) -8(\alpha+1)|\lambda|^2 r H(r)-4|\lambda|^3 r^2 H(r)-2|\lambda|^2r^2H(r)
			-2(n_1-2)(\alpha+1)|\lambda| H(r)\\
			&-4(\alpha+1)|\lambda| H(r)-2(\alpha+1)|\lambda| r H^\prime(r)-2(\alpha+1)|\lambda|^2 rH(r)\\
			= &-4|\lambda| I(r)-(10(\alpha+1)|\lambda|^2 r+4|\lambda|^3r^2+2|\lambda|^2 r^2+2n_1(\alpha+1)|\lambda|) H(r) \\
			&-2(\alpha+1)|\lambda| r\left(\frac{2\alpha+n_1}{r}H(r)+\frac{1}{r(\alpha+1)}I(r)\right)\\
			=&-6|\lambda| I(r)-(10(\alpha+1)|\lambda|^2 r+4|\lambda|^3r^2+2|\lambda|^2 r^2+2n_1(\alpha+1)|\lambda|\\
            &+2(\alpha+1)(2\alpha+n_1)|\lambda|) H(r).
		\end{align*}
		
		Therefore, we have
		\begin{align*}
			& I^\prime(r)\\
            \geq & \frac{2\alpha+n_1}{r}I(r)+\frac{4(\alpha+1)}{r}\sum_A\int \langle \del u_A,x\rangle^2 (r^2-|x|^2)^\alpha dx-6\lambda I(r)\\
			&-(10(\alpha+1)|\lambda|^2 r+4|\lambda|^3r^2+2|\lambda|^2 r^2+2n_1(\alpha+1)|\lambda|+2(\alpha+1)(2\alpha+n_1)|\lambda|) H(r).
		\end{align*}
		Consequently,
		\begin{align*}
			N^\prime(r) = &\frac{1}{H^2}\left(I^\prime(r)H(r)-I(r)H^\prime(r)\right)\\
			 \geq &\frac{1}{H^2}\frac{4(\alpha+1)}{r}\left(H\sum_A\int \langle \del u_A,x\rangle^2 (r^2-|x|^2)^\alpha dx -\frac{1}{4(\alpha+1)^2}I^2\right)\\
			&-6|\lambda| N(r)-(10(\alpha+1)|\lambda|^2 r+4|\lambda|^3r^2+2|\lambda|^2 r^2+4(\alpha+1)(\alpha+n_1)|\lambda|).
		\end{align*}
		Let $K=\sum_A K_A=\sum_A \int \langle \del u_A,x\rangle^2(r^2-|x|^2)^\alpha dx$. By Cauchy-Schwarz's inequality, we have
		\begin{align*}
			\frac{1}{4(\alpha+1)^2}I_A^2=&\left(\int \langle x,\del u_A\rangle u_A(r^2-|x|^2)^\alpha\right)^2\\
			\leq &\int \langle x,\del u_A\rangle^2 (r^2-|x|^2)^\alpha dx \int u_A^2 (r^2-|x|^2)^\alpha dx\\ 
			=& K_A H_A.
		\end{align*}
		Consequently,
		\begin{align*}
			\frac{1}{4(\alpha+1)^2}I^2 & = \frac{1}{4(\alpha+1)^2}\left(\sum_A I_A^2 +2\sum_{|A|<|B|}I_AI_B\right)\\
			&\leq \frac{1}{4(\alpha+1)^2}\left(\sum_A I_A^2 +2\sum_{|A|<|B|}|I_A| |I_B|\right)\\
			&\leq \sum_A K_A H_A +2\sum_{|A|<|B|}\sqrt{K_AH_AK_BH_B}\\
			&\leq \sum_A K_A H_A +\sum_{|A|<|B|}(K_AH_B+K_BH_A)\\
			&= (\sum_A K_A)(\sum_B H_B)\\
			&= KH.
		\end{align*}
		
		Therefore,
		\begin{align*}
			N^\prime(r) \geq  -6|\lambda| N(r)-(10(\alpha+1)|\lambda|^2 r+4|\lambda|^3r^2+2|\lambda|^2 r^2+4(\alpha+1)(\alpha+n_1)|\lambda|),
		\end{align*}
		which implies that $e^{6|\lambda| r}(N(r)+p(r)) $ is nondescreasing for $\lambda\neq 0,$ where
		$p(r)=ar^2+br+c$ with $a=\frac{2|\lambda|^2+|\lambda|}{3},b=\frac{5(\alpha+1)}{3}|\lambda|-\frac{2|\lambda|+1}{9}$ and $c=\frac{2(\alpha+1)(\alpha+n_1)}{3}-\frac{5(\alpha+1)}{18}+\frac{2|\lambda|+1}{54|\lambda|}.$
		
		When $\lambda =0,$ it becomes $Du=0$, which means that $u$ is monogenic. The estimate of $I^\prime(r)$ is given as follows, i.e.,
		\begin{align*}
			I^\prime(r)\geq \frac{2\alpha+n_1}{r} I(r)+\frac{4(\alpha+1)}{r}\sum_A\int \langle \del u_A, x\rangle^2 (r^2-|x|^2)^\alpha dx,
		\end{align*}
		and consequently,
		\begin{align*}
			N^\prime(r) &= \frac{1}{H^2} (I^\prime(r)H(r)-I(r)H^\prime(r))\\
			& \geq \frac{1}{H^2} \frac{4(\alpha+1)}{r}\left(H\sum_A\int \langle \del u_A,x\rangle^2(r^2-|x|^2)^\alpha dx-\frac{1}{4(\alpha+1)^2}I^2\right)\\
			&\geq 0.
		\end{align*} 
		Thus $N(r)$ is nondescreasing.
	\end{proof}
	\section{Three balls inequality}
	In this section we will prove the main results.
	
	\noindent{\textbf{Proof of Theorem \ref{three-balls-L2}.}
		Define 
		\begin{align*}
			h(r)=\int_{B_r}|u(x)|^2 dx,
		\end{align*}
		where $B_r$ is the ball centering at the origin with radius $r$. 
		There hold
		\begin{align}\label{h1}
			H(r)\leq r^{2\alpha} h(r)
		\end{align}
		and
		\begin{align}\label{h2}
			H(2r)\geq 3^\alpha r^{2\alpha} h(r),
		\end{align}
		where (\ref{h2}) follows from the fact
		\begin{align*}
			h(r)=\int_{B_r}|u|^2 dx \leq \int_{B_{r}}|u|^2(1+\frac{r^2-|x|^2}{4r^2-r^2})^\alpha dx\leq \int_{B_{2r}}|u|^2(1+\frac{r^2-|x|^2}{4r^2-r^2})^\alpha=\frac{H(2r)}{3^\alpha r^{2\alpha}}.
		\end{align*}
		By (\ref{Hr}), we have that
		\begin{align*}
			\frac{H^\prime(r)}{H(r)}=\frac{2\alpha+n_1}{r} + \frac{N(r)}{r(\alpha+1)}.
		\end{align*}
		Integrating the above equality from $r_1$ to $2r_2$, we have
		\begin{align*}
			\log \frac{H(2r_2)}{H(r_1)} =& \int_{r_1}^{2r_2} \frac{2\alpha+n_1}{r} dr + \int_{r_1}^{2r_2}\frac{N(r)}{r(\alpha+1)}dr\\
			=& (2\alpha+n_1)\log \frac{2r_2}{r_1} +\int_{r_1}^{2r_2}\frac{e^{6\lambda r}(N(r)+p(r))}{r(\alpha+1)e^{6\lambda r}}dr -\int_{r_1}^{2r_2} \frac{ar^2+br+c}{r(\alpha+1)}dr.
		\end{align*}
		Thus
		\begin{align*}
			\log \frac{H(2r_2)}{H(r_1)}\leq &(2\alpha+n_1)\log \frac{2r_2}{r_1} + \frac{e^{12\lambda r_2}(N(2r_2)+p(2r_2))}{(\alpha+1)e^{6\lambda r_1}}\log \frac{2r_2}{r_1} -\frac{c}{\alpha+1}\log \frac{2r_2}{r_1}\\
			& -\frac{1}{\alpha+1}\left(\frac{a}{2}((2r_2)^2-r_1^2)+b(2r_2-r_1)\right).
		\end{align*}
		Similarly, integrating from $2r_2$ to $r_3$, we have
		\begin{align*}
			\log \frac{H(r_3)}{H(2r_2)} &= \int_{2r_2}^{r_3} \frac{2\alpha+n_1}{r} dr + \int_{2r_2}^{r_3}\frac{N(r)}{r(\alpha+1)} dr\\
			&=(2\alpha+n_1)\log \frac{r_3}{2r_2} + \int_{2r_2}^{r_3} \frac{e^{6\lambda r}(N(r)+p(r)}{r(\alpha+1)e^{6\lambda r}} dr -\int_{2r_2}^{r_3} \frac{ar^2+br+c}{r(\alpha+1)}dr.
		\end{align*}
		Then,
		\begin{align*}
			\log\frac{H(r_3)}{H(2r_2)}\geq & (2\alpha+n_1)\log\frac{r_3}{2r_2} +\frac{e^{12\lambda r_2}(N(2r_2)+p(2r_2)}{(\alpha+1)e^{6\lambda r_3}}\log \frac{r_3}{2r_2} - \frac{c}{\alpha+1}\log \frac{r_3}{2r_2}\\
			&-\frac{1}{\alpha+1}\left(\frac{a}{2}(r_3^2-(2r_2)^2)+b(r_3-2r_2)\right).
		\end{align*}
		Hence,
		\begin{align*}
			&\frac{\log \frac{H(2r_2)}{H(r_1)}}{\log \frac{2r_2}{r_1}} + \frac{c}{\alpha+1}-(2\alpha+n_1)+\frac{1}{\log\frac{2r_2}{r_1}(\alpha+1)}\left(\frac{a}{2}((2r_2)^2-r_1^2)+b(2r_2-r_1)\right)\\
			\leq &\frac{\log \frac{H(r_3)}{H(2r_2)}}{\log \frac{r_3}{2r_2}}+ \frac{c}{\alpha+1} -(2\alpha+n_1)+\frac{1}{\log\frac{r_3}{2r_2}(\alpha+1)}\left(\frac{a}{2}(r_3^2-(2r_2)^2)+b(r_3-2r_2)\right),
		\end{align*}
		and consequently,
		\begin{align*}
			\log \left(\frac{H(2r_2)}{H(r_1)}\right)^{\frac{1}{\log\frac{2r_2}{r_1}}}\leq &\log \left(\frac{H(r_3)}{H(2r_2)}\right)^{\frac{1}{\log \frac{r_3}{2r_2}}} +\frac{1}{\log\frac{r_3}{2r_2}(\alpha+1)}\left(\frac{a}{2}(r_3^2-(2r_2)^2+b(r_3-2r_2)\right)\\
			&-\frac{1}{\log\frac{2r_2}{r_1}(\alpha+1)}\left(\frac{a}{2}((2r_2)^2-r_1^2)+b(2r_2-r_1)\right),
		\end{align*}
		which is equivalent to
		\begin{align*}
			&(H(2r_2))^{\frac{1}{\log\frac{2r_2}{r_1}}+\frac{1}{\log \frac{r_3}{2r_2}}} \\
         \leq & e^{\frac{\left(\frac{a}{2}(r_3^2-(2r_2)^2)+b(r_3-2r_2)\right)}{(\alpha+1)\log\frac{r_3}{2r_2}}-\frac{\left(\frac{a}{2}((2r_2)^2-r_1^2)+b(2r_2-r_1)\right)}{(\alpha+1)\log\frac{2r_2}{r_1}}} (H(r_1))^{\frac{1}{\log\frac{2r_2}{r_1}}}(H(r_3))^{\frac{1}{\log \frac{r_3}{2r_2}}}.
		\end{align*}
		Combining the above inequality with (\ref{h1}) and (\ref{h2}), we have that
		\begin{align*}
			&3^{\alpha (C_1+C_2)}r_2^{2\alpha(C_1+C_2)}(h(r_2))^{C_1+C_2}\\
			\leq & e^{\frac{\left(\frac{a}{2}(r_3^2-(2r_2)^2)+b(r_3-2r_2)\right)}{(\alpha+1)C_2}-\frac{\left(\frac{a}{2}((2r_2)^2-r_1^2)+b(2r_2-r_1)\right)}{(\alpha+1)C_1}}r_1^{2C_1\alpha}r_3^{2C_2\alpha} (h(r_1))^{C_1}(h(r_3))^{C_2}.
		\end{align*}
		This implies that
		\begin{align*}
			\int_{B_{r_2}} |u|^2 dx \leq C_3 \left(\int_{B_{r_1}} |u|^2dx\right)^{\frac{C_1}{C_1+C_2}} \left(\int_{B_{r_3}} |u|^2dx\right)^{\frac{C_2}{C_1+C_2}},
		\end{align*}
		where $C_1=\frac{1}{\log\frac{2r_2}{r_1}}, C_2=\frac{1}{\log \frac{r_3}{2r_2}}$ and $$C_3=\frac{r_1^{2\alpha\frac{C_1}{C_1+C_2}}r_3^{2\alpha\frac{C_2}{C_1+C_2}}}{3^{\alpha}r_2^{2\alpha}}e^{\frac{\left(\frac{a}{2}(r_3^2-(2r_2)^2)+b(r_3-2r_2)\right)}{(\alpha+1)C_2(C_1+C_2)}-\frac{\left(\frac{a}{2}((2r_2)^2-r_1^2)+b(2r_2-r_1)\right)}{(\alpha+1)C_1(C_1+C_2)}}.$$
		
		In particular, when $\lambda=0,$ we have that
		\begin{align*}
			H(2r_2)^{\frac{1}{\log\frac{2r_2}{r_1}}+\frac{1}{\log\frac{r_3}{2r_2}}}\leq H(r_1)^\frac{1}{\log \frac{2r_2}{r_1}} H(r_3)^\frac{1}{\log \frac{r_3}{2r_2}}.
		\end{align*} 
		Consequently, the three balls inequality becomes
		\begin{align*}
			\int_{B_{r_2}} |u|^2 dx \leq C_4\left(\int_{B_{r_1}}|u|^2dx\right)^\frac{C_1}{C_1+C_2}\left(\int_{B_{r_3}}|u|^2dx\right)^\frac{C_2}{C_1+C_2},
		\end{align*}
		where $C_1=\frac{1}{\log\frac{2r_2}{r_1}},$ $C_2=\frac{1}{\log \frac{r_3}{2r_2}}$ and $C_4=\frac{r_1^{2\alpha\frac{C_1}{C_1+C_2}}r_3^{2\alpha\frac{C_2}{C_1+C_2}}}{3^{\alpha}r_2^{2\alpha}}.$
		\endproof

\bigskip 
  
		\noindent{\textbf{Proof of Corollary \ref{three-spheres-infty}.}}
		When $\lambda=0,$ we can easily deduce the three balls inequality in the $L^\infty$-norm from Theorem \ref{three-balls-L2}. In fact, by the subharmonicity of $|u|^2$ we have that, for $x\in B_\rho,$
		\begin{align*}
			|u(x)|^2 \leq \frac{\Gamma(\frac{n_1}{2}+1)}{\pi^{\frac{n_1}{2}}r^{n_1}} \int_{B_r(x)} |u(y)|^2 dy\leq \frac{\Gamma(\frac{n_1}{2}+1)}{\pi^{\frac{n_1}{2}}r^{n_1}} \int_{B_{r+\rho}} |u(y)|^2 dy,
		\end{align*}
		where $\Gamma(\cdot)$ is the Gamma function. 
		This implies that
		\begin{align*}
			||u||_{L^\infty(B_\rho)}\leq \left(\frac{\Gamma(\frac{n_1}{2}+1)}{\pi^\frac{n_1}{2}}\right)^\frac{1}{2} r^{-\frac{n_1}{2}} ||u||_{L^2(B_{r+\rho})},
		\end{align*}
		or equivalently,
		\begin{align*}
			||u||_{L^\infty(B_\rho)}\leq \left(\frac{\Gamma(\frac{n_1}{2}+1)}{\pi^\frac{n_1}{2}}\right)^\frac{1}{2}(\delta-\rho)^{-\frac{n_1}{2}}||u||_{L^2(B_\delta)}.
		\end{align*}
		Since $r_3> 2r_2>r_2>r_1,$ we have
		\begin{align*}
			||u||_{L^\infty(B_{r_2})}\leq \left(\frac{\Gamma(\frac{n_1}{2}+1)}{\pi^\frac{n_1}{2}}\right)^\frac{1}{2}\left(\frac{r_3+r_2}{3}-r_2\right)^{-\frac{n_1}{2}}||u||_{L^2(B_\frac{r_3+r_2}{3})}.
		\end{align*}
		Note that $r_3>\frac{2(r_3+r_2)}{3}>\frac{r_3+r_2}{3}>r_2>r_1.$
		Then, we have
		\begin{align*}
			||u||_{L^2(B_\frac{r_3+r_2}{3})}\leq C_4^\prime ||u||_{L^2({B_{r_1}})}^{\frac{C_1^\prime}{C_1^\prime+C_2^\prime}} ||u||_{L^2(B_{r_3})}^{\frac{C_2^\prime}{C_1^\prime+C_2^\prime}},
		\end{align*}
		where $C_1^\prime=\frac{1}{\log \frac{2(r_3+r_2)}{3r_1}},$ $C_2^\prime= \frac{1}{\log \frac{3r_3}{2(r_3+r_2)}}$ and $C_4^\prime=\frac{3^\alpha r_1^{2\alpha\frac{C_1}{C_1+C_2}}r_3^{2\alpha\frac{C_2}{C_1+C_2}}}{4^\alpha((r_3+r_2))^{2\alpha}}.$
  
		Hence, 
		\begin{align*}
			&||u||_{L^\infty(B_{r_2})} \\
            \leq & 3^\frac{n_1}{2}\left(\frac{\Gamma(\frac{n_1}{2}+1)}{\pi^\frac{n_1}{2}}\right)^\frac{1}{2}C_4^\prime (r_3-2r_2)^{-\frac{n_1}{2}}||u||_{L^2({B_{r_1}})}^{\frac{C_1^\prime}{C_1^\prime+C_2^\prime}} ||u||_{L^2(B_{r_3})}^{\frac{C_2^\prime}{C_1^\prime+C_2^\prime}}\\
			\leq & 3^\frac{n}{2}\left(\frac{\Gamma(\frac{n}{2}+1)}{\pi^\frac{n}{2}}\right)^\frac{1}{2}C_4^\prime (r_3-2r_2)^{-\frac{n}{2}}|B_{r_1}|^\frac{C_1^\prime}{2(C_1^\prime+C_2^\prime)}|B_{r_3}|^\frac{C_2^\prime}{2(C_1^\prime+C_2^\prime)}||u||_{L^\infty({B_{r_1}})}^{\frac{C_1^\prime}{C_1^\prime+C_2^\prime}} ||u||_{L^\infty(B_{r_3})}^{\frac{C_2^\prime}{C_1^\prime+C_2^\prime}}\\
			\leq & 3^\frac{n_1}{2}\left(\frac{\Gamma(\frac{n_1}{2}+1)}{\pi^\frac{n_1}{2}}\right)^\frac{1}{2}C_4^\prime (r_3-2r_2)^{-\frac{n_1}{2}}\left(\frac{\pi^\frac{n_1}{2}}{\Gamma(\frac{n_1}{2}+1)}r_3^{n_1}\right)^\frac{1}{2}||u||_{L^\infty({B_{r_1}})}^{\frac{C_1^\prime}{C_1^\prime+C_2^\prime}} ||u||_{L^\infty(B_{r_3})}^{\frac{C_2^\prime}{C_1^\prime+C_2^\prime}}\\
			\leq & 3^\frac{n_1}{2} C_4^\prime (r_3-2r_2)^{-\frac{n_1}{2}}r^\frac{n_1}{2}_3||u||_{L^\infty({B_{r_1}})}^{\frac{C_1^\prime}{C_1^\prime+C_2^\prime}} ||u||_{L^\infty(B_{r_3})}^{\frac{C_2^\prime}{C_1^\prime+C_2^\prime}}.
		\end{align*}
		\endproof

        To prove Corollary \ref{three-spheres-infty1}, we need the following lemma.
        \begin{lemma}\label{Moser}
			Suppose that $Du=\lambda u$ with $\lambda\neq 0.$ Then, for $0<r<R<1,$ there holds
			\begin{align*}
				||u||_{L^\infty(B_r)} \leq \widetilde M (R-r)^{-\frac{n_1}{2}} ||u||_{L^2(B_R)},
			\end{align*}
			where $\widetilde M$ is a positive constant.
		\end{lemma}
		
		\begin{proof}
			Firstly, for all possible index $A,$ we note that
			\begin{align*}
				\laplace u_A= 2\lambda\partial_0 u_A-\lambda^2 u_A.
			\end{align*}
			By using the standard Moser iteration method for such $u_A$ (see e.g. \cite{Han-Lin}), there exists a positive constant $\widetilde M$ such that
			\begin{align*}
				||u_A||_{L^\infty(B_r)} \leq \widetilde M (R-r)^{-\frac{n_1}{2}} ||u_A||_{L^2(B_R)}.
			\end{align*}
			Therefore,
			\begin{align*}
				||u||_{L^\infty(B_r)} \leq \widetilde M (R-r)^{-\frac{n_1}{2}} ||u||_{L^2(B_R)}.
			\end{align*}
        \end{proof}
        \bigskip
        
		\noindent{\textbf{Proof of Corollary \ref{three-spheres-infty1}.}}
		By Lemma \ref{Moser}, we have that, for $0<r_1<r_2<2r_2<r_3<1,$
		\begin{align*}
			||u||_{L^\infty(B_{r_2})}\leq \widetilde M\left(\frac{r_3+r_2}{3}-r_2\right)^{-\frac{n_1}{2}}||u||_{L^2(B_\frac{r_3+r_2}{3})},
		\end{align*}
		where $\widetilde M$ is a positive constant given in Lemma \ref{Moser}.
		Note that $1>r_3>\frac{2(r_3+r_2)}{3}>\frac{r_3+r_2}{3}>r_2>r_1.$
		Then, applying Theorem \ref{three-balls-L2}, we have that
		\begin{align*}
			||u||_{L^2(B_\frac{r_3+r_2}{3})}\leq C_3^\prime ||u||_{L^2({B_{r_1}})}^{\frac{C_1^\prime}{C_1^\prime+C_2^\prime}} ||u||_{L^2(B_{r_3})}^{\frac{C_2^\prime}{C_1^\prime+C_2^\prime}},
		\end{align*}
		where $C_1^\prime=\frac{1}{\log \frac{2(r_3+r_2)}{3r_1}},$ $C_2^\prime= \frac{1}{\log \frac{3r_3}{2(r_3+r_2)}}$\\
		\noindent and $C_3^\prime=\frac{3^\alpha r_1^{2\alpha\frac{C_1}{C_1+C_2}}r_3^{2\alpha\frac{C_2}{C_1+C_2}}}{4^\alpha((r_3+r_2))^{2\alpha}}e^{\frac{\left(\frac{a}{2}(r_3^2-(\frac{2(r_3+r_2)}{3})^2)+b(r_3-\frac{2(r_3+r_2)}{3})\right)}{(\alpha+1)C_2(C_1+C_2)}-\frac{\left(\frac{a}{2}((\frac{2(r_3+r_2)}{3})^2-r_1^2)+b(\frac{2(r_3+r_2)}{3}-r_1)\right)}{(\alpha+1)C_1(C_1+C_2)}}.$
		
		Therefore, 
		\begin{align*}
			||u||_{L^\infty(B_{r_2})} &\leq 3^\frac{n_1}{2}\widetilde MC_3^\prime (r_3-2r_2)^{-\frac{n_1}{2}}||u||_{L^2({B_{r_1}})}^{\frac{C_1^\prime}{C_1^\prime+C_2^\prime}} ||u||_{L^2(B_{r_3})}^{\frac{C_2^\prime}{C_1^\prime+C_2^\prime}}\\
			&\leq  3^\frac{n_1}{2}\widetilde MC_3^\prime (r_3-2r_2)^{-\frac{n_1}{2}}|B_{r_1}|^\frac{C_1^\prime}{2(C_1^\prime+C_2^\prime)}|B_{r_3}|^\frac{C_2^\prime}{2(C_1^\prime+C_2^\prime)}||u||_{L^\infty({B_{r_1}})}^{\frac{C_1^\prime}{C_1^\prime+C_2^\prime}} ||u||_{L^\infty(B_{r_3})}^{\frac{C_2^\prime}{C_1^\prime+C_2^\prime}}\\
			&\leq  M^\prime C_3^\prime (r_3-2r_2)^{-\frac{n_1}{2}}r^\frac{n_1}{2}_3||u||_{L^\infty({B_{r_1}})}^{\frac{C_1^\prime}{C_1^\prime+C_2^\prime}} ||u||_{L^\infty(B_{r_3})}^{\frac{C_2^\prime}{C_1^\prime+C_2^\prime}}.
		\end{align*}


\begin{thebibliography}{00}
			\bibitem{Alessandrini-Rondi-Rosset-Vessella} G. Alessandrini, L. Rondi, E. Rosset and S. Vessella, {\it The stability for the Cauchy problem for elliptic equations,} Inverse Problem, {\bf 25} (2009), 123004, 47 pp.	
			
			
			\bibitem{Almgren} F. J. Almgren, {\it Dirichlet's problem for multiple valued functions and the regularity of mass minimizing integral currents,} Minimal submanifolds and geodesics (Proc. Japan-United States Sem., Tokyo, 1977), North-Holland, Amsterdam-New York, 1979, 1-6.
			
			\bibitem{Abul-Constales-Morais-Zayed} M. Abu-Ez, D. Constales, J. Morais and M. Zzyed, {\it Hadamard three-hyperballs types theorem and overconvergence of special monogenic simple series,} J. Math. Anal. Appl., {\bf 412} (2014), 426-434.
			
			
			\bibitem{Brackx-Delanghe-Sommen} F. Brackx, R. Delanghe and F. Sommen, {\it Clifford analysis}, vol 76, Pitman Books Limited, London, 1982.
			
			\bibitem{Colding-Minicozzi} T. Colding and W. Minicozzi II, {\it Harmonic functions with polynomial growth,} J. Differential Geom., {\bf 45} (1997), 1-77.
			
			\bibitem{Duren} P. L. Duren, {\it Theory of $H^p$ spaces,} Pure and Applied Mathematics, Volume 38, Academic Press, New York and London, 1970.
			
			\bibitem{Garofalo-Lin} N. Garofalo and F.-H. Lin, {\it Monotonicity properties of variational inegrals, $A_p$ weights and unique continuation,} Indiana Univ. Math. J., {\bf 35} (1986), 245-268.
			
			\bibitem{Garofalo-Lin1} N. Garofalo and F.-H. Lin, {\it Unique continuation for elliptic operators: a geometric-variational approach,} Comm. Pure Appl. Math., {\bf 40} (1987), 347-366.
			
			\bibitem{Hadamard} J. Hadamard, {\it Sur les fonctions entieres,} Bull. Soc. Math. France, {\bf 24} (1896), 186-187.
			
			\bibitem{Han-Lin} Q. Han and F. Lin, {\em Elliptic partial differential equations, second edition,} American Mathematical Society, Providence, Rhode Island, 2011. 
			
			\bibitem{Liu} G. Liu, {\it Three circle theorems and dimension estimate for holomorphic functions on K\"ahler manifolds,} Duke Math. J., {\bf 165} (2016), 2899-2919.


           \bibitem{Ou} J. Ou, \textit{Three circle theorem of eigenfunctions on noncompact shrinking Ricci solitons with constant scalar curvature},  J. Math. Anal. Appl., {\bf 464} (2018), 1243-1259.
   
			\bibitem{Peetre-Sjolin} J. Peetre and P. Sj\" olin, {\it Three-line theorem and Clifford analysis,} Complex Variables Theory Appl., {\bf 19} (1992), 151-163.
			
			\bibitem{Protter-Weinberger} M. H. Protter and H. F. Weinberger, {\it Maximum principles in differential equations,} Springer-Verlag, New York, 1984.
			
			
			\bibitem{Ullrich} D. C. Ullrich, {\it Complex made simple,} Graduate Studies in Mathematics 97, American Mathematical Society, 2008.
			
			\bibitem{Xu} G. Xu, {\it Three circles theorems for harmonic functions,} Math. Ann., {\bf 366} (2016), 1281-1317.
			\bibitem{Zhu} J. Zhu, {\it Quantitative uniqueness of elliptic equations,} Amer. J. Math., {\bf 138} (2016), 733-762.
		\end{thebibliography}
	\end{document}